\documentclass[11pt]{amsart}
\usepackage{amssymb,amsthm}
\usepackage[nosumlimits]{mathtools}
\usepackage{enumerate,tabu}
\usepackage[cal=euler,scr=rsfso]{mathalpha}
\usepackage[margin=1in]{geometry}
\usepackage{quiver}
\usepackage{stmaryrd}
\usepackage{upgreek}

\theoremstyle{plain}
\newtheorem{theorem}{Theorem}[section]
\newtheorem{proposition}[theorem]{Proposition}
\newtheorem{lemma}[theorem]{Lemma}
\newtheorem{corollary}[theorem]{Corollary}

\theoremstyle{definition}
\newtheorem{definition}[theorem]{Definition}
\newtheorem{example}[theorem]{Example}

\theoremstyle{remark}

\usepackage{hyperref}
\usepackage{xcolor}
\hypersetup{
    colorlinks,
    linkcolor={red!50!black},
    citecolor={blue!50!black},
    urlcolor={blue!80!black}
}

\newcommand{\tp}{{\scriptscriptstyle\mathsf{T}}}
\newcommand{\p}{{\scriptscriptstyle+}}
\newcommand{\pp}{{\scriptscriptstyle++}}

\newcommand{\lb}{\llbracket}
\newcommand{\rb}{\rrbracket}

\let\latexcirc=\circ
\newcommand{\ccirc}{\mathbin{\mathchoice
  {\xcirc\scriptstyle}
  {\xcirc\scriptstyle}
  {\xcirc\scriptscriptstyle}
  {\xcirc\scriptscriptstyle}
}}
\newcommand{\xcirc}[1]{\vcenter{\hbox{$#1\latexcirc$}}}
\let\circ\ccirc

\let\O\undefined
\let\P\undefined
\DeclareMathOperator{\O}{O}
\DeclareMathOperator{\GL}{GL}
\DeclareMathOperator{\P}{P}
\DeclareMathOperator{\V}{V}
\DeclareMathOperator{\St}{St}

\DeclareMathOperator{\im}{im}

\DeclareMathOperator{\conv}{conv}
\DeclareMathOperator{\tr}{tr}
\DeclareMathOperator{\Gr}{Gr}
\DeclareMathOperator{\El}{\mathcal{E}}
\DeclareMathOperator{\diag}{diag}
\DeclareMathOperator{\rank}{rank}

\begin{document}
\title{Grassmannian optimization is NP-hard}
\author[Z.~Lai]{Zehua~Lai}
\address{Department of Mathematics, University of Texas, Austin, TX 78712}
\email{zehua.lai@austin.utexas.edu}
\author[L.-H.~Lim]{Lek-Heng~Lim}
\address{Computational and Applied Mathematics Initiative, Department of Statistics,
University of Chicago, Chicago, IL 60637}
\email{lekheng@uchicago.edu}
\author[K.~Ye]{Ke Ye}
\address{KLMM, Academy of Mathematics and Systems Science, Chinese Academy of Sciences, Beijing 100190, China}
\email{keyk@amss.ac.cn}

\begin{abstract}
We show that unconstrained quadratic optimization over a Grassmannian $\Gr(k,n)$ is NP-hard. Our results cover all scenarios: (i) when $k$ and $n$ are both allowed to grow; (ii) when $k$ is arbitrary but fixed; (iii) when $k$ is fixed at its lowest possible value $1$. We then deduce the NP-hardness of unconstrained cubic optimization over the Stiefel manifold $\V(k,n)$ and the orthogonal group $\O(n)$. As an addendum we demonstrate the NP-hardness of unconstrained quadratic optimization over the Cartan manifold, i.e., the positive definite cone $\mathbb{S}^n_\pp$ regarded as a Riemannian manifold, another popular example in manifold optimization. We will also establish the nonexistence of $\mathrm{FPTAS}$ in all cases.
\end{abstract}
\maketitle

\section{Introduction}\label{sec:intro}

It is well-known that unconstrained quadratic programming is easy over the $n$-space $\mathbb{R}^n$ or $n$-sphere $\mathrm{S}^n$. The former is a simple exercise\footnote{Minimum of $\frac12 x^\tp A x - b^\tp x$ is either attained at $x_* =A^{-1}b$ if $A$ is positive definite or else it is $-\infty$. The common refrain ``QP is NP-hard'' \cite{vavaQP} refers to constrained QP  over $\mathbb{R}^n$.}  whereas the latter is polynomial-time solvable \cite[Section~4.3]{vava}.
It is therefore somewhat surprising that optimizing a quadratic polynomial over the Grassmannian $\Gr(k,n)$ is an NP-hard problem, not least because the $k =1$ case is intimately related to $\mathbb{R}^n$ and $\mathrm{S}^n$: $\Gr(1,n+1) = \mathbb{RP}^n$ is a compactification of $\mathbb{R}^n$ and $\mathrm{S}^n$ is a double-cover of $\Gr(1,n+1)$.  We will prove three variations of this result:
\begin{center}
\small
\tabulinesep=0.75ex
\begin{tabu}{@{}lll}
\textsc{assumption} & \textsc{reduction} & \textsc{section}\\\tabucline[0.5pt ] -
$k$ and $n$ can both grow & clique decision & Section~\ref{sec:both}\\
$k$ arbitrary but fixed, only $n$ can grow & clique number  & Section~\ref{sec:onlyn}\\
$k = 1$, only $n$ can grow & stability number & Section~\ref{sec:stie}
\end{tabu}
\end{center}
We first establish these using the representation $\Gr(k,n) = \{P \in \mathbb{R}^{n \times n}: P^2 = P = P^\tp,\; \tr(P) = k\}$, and then extend them to all known models of the Grassmannian.

These hardness results for the Grassmannian form the mainstay of our article, accounting for the choice of title. But by leveraging on them, we will also deduce the NP-hardness of optimization problems over some of the most common manifolds in applied mathematics:
\begin{center}
\small
\tabulinesep=0.75ex
\begin{tabu}{@{}lll}
\textsc{manifold} & \textsc{hardness result} & \textsc{reduction}\\\tabucline[0.5pt ] -
Stiefel manifold $\V(k,n)$ & unconstrained cubic programming is NP-hard & QP over $\Gr(k,n)$\\
Orthogonal group $\O(n)$ & unconstrained cubic programming is NP-hard & CP over $\V(k,n)$ \\
Cartan manifold $\mathbb{S}^n_\pp$ & unconstrained quadratic programming is NP-hard & matrix copositivity
\end{tabu}
\end{center}
Here the Stiefel manifold is modeled as the set of $n \times k$ matrices with orthonormal columns $\V(k,n)$  \cite{Stie}, and the Cartan manifold is modeled as the positive definite cone $\mathbb{S}^n_\pp$ but equipped with the Riemannian metric $\mathsf{g}_A (X,Y) = \tr(A^{-1} X A^{-1} Y)$ for $A \in \mathbb{S}^n_\pp$ \cite{Cartan1}. Note that as $\V(1,n+1) = \mathrm{S}^n$ and unconstrained quadratic programming on the $n$-sphere is polynomial-time, we need at least cubic programming (CP) to demonstrate NP-hardnesss over $\V(k,n)$ for all possible $k,n$.

\subsection*{Notations} We write $\mathrm{S}^n$ for the $n$-sphere and $\mathbb{S}^n$ for the set of $n \times n$ symmetric matrices. We write $\mathbb{S}^n_\p$ and $\mathbb{S}^n_\pp$ for the positive semidefinite and positive definite cones respectively. The operation $\diag : \mathbb{R}^{n \times n} \to \mathbb{R}^n$ takes a matrix $X$ to its vector of diagonal entries $\diag(X)$. We write $\O(n)\coloneqq \{Q \in \mathbb{R}^{n \times n} :  Q^\tp Q = I \}$ for the orthogonal group and $\GL(n) \coloneqq \{ X \in \mathbb{R}^{n \times n} : \rank(X) = n\}$ for the general linear group. A matrix group is a subgroup of $\GL(n)$ for some $n \in \mathbb{N}$. A matrix submanifold is a submanifold of $\mathbb{R}^{n \times k}$ for some $n,k \in \mathbb{N}$ and  a matrix quotient manifold is one that is given by the quotient of a matrix submanifold by a matrix group; collectively we refer to both as matrix manifolds. We identify $\mathbb{R}^n = \mathbb{R}^{n \times 1}$, i.e., an $x \in \mathbb{R}^n$ will always be a column vector; row vectors will be written $x^\tp$. We delimit equivalence classes with double brackets $\lb \, \cdot \, \rb$.

\section{Models of the Grassmann, Stiefel, and Cartan manifolds}\label{sec:mod}

As \emph{abstract manifolds}, the Grassmannian is the set of $k$-dimensional subspaces in $\mathbb{R}^n$, the noncompact Stiefel manifold is the set of $k$-frames in $\mathbb{R}^n$, the compact Stiefel manifold is the set of orthonormal $k$-frames in $\mathbb{R}^n$, and the Cartan manifold is the set of ellipsoids in $\mathbb{R}^n$ centered at origin. We denote them by $\Gr_k(\mathbb{R}^n)$, $\St_k(\mathbb{R}^n)$, $\V_k(\mathbb{R}^n)$, $\El(\mathbb{R}^n)$ respectively, following standard notations in geometry and topology, whenever we want to refer to them  as abstract manifolds.

To do anything with these manifolds, one needs a concrete \emph{model} that parameterizes points on the manifold, possibly also vectors on tangent spaces, that is convenient to work with. Put in another way, a model endows the manifold with a system of extrinsic coordinates. In the context of this article, without a model it would not even make sense to speak of quadratic or cubic programming, as the degree of a polynomial function invariably depends on the coordinates used to parameterize points on the manifold.

The models favored by practitioners, particularly those in manifold optimization, are typically chosen with a view towards computational feasibility, and basically fall into two distinct families:
\begin{enumerate}[\upshape (i)]
\item \emph{quotient models}:  represent points and tangent vectors as equivalence classes of matrices;
\item \emph{submanifold models}: represent points and tangent vectors as actual matrices.
\end{enumerate}
As we will see in Section~\ref{sec:other}, the NP-hardness claimed in the previous section will extend to every model.

The quotient models are matrix quotient manifolds $\mathcal{M} / G$ obtained from a matrix group $G$ acting on a matrix submanifold $\mathcal{M}$ through matrix multiplication. Table~\ref{tab:quo} lists all quotient models for the  Stiefel, Grassmann, and Cartan manifolds that we are aware of:
\begin{table*}[h]
\small
\caption{Quotient models}
\label{tab:quo}
\tabulinesep=0.75ex
\begin{tabu}{@{}lll}
$\mathcal{M}$ & $G$ & $\mathcal{M}/G$ \\\tabucline[0.5pt ] -
$\O(n)$ & $\O(k)\times \O(n-k)$ & $\Gr_k(\mathbb{R}^n)$ \\
$\V(k,n)$ & $\O(k)$  & $\Gr_k(\mathbb{R}^n)$ \\
$\GL(n)$ & $\P(k,n)$  & $\Gr_k(\mathbb{R}^n)$ \\
$\St(k,n)$ & $\GL(k)$  & $\Gr_k(\mathbb{R}^n)$ \\
$\O(n)$ & $\O(n-k)$ & $\V_k(\mathbb{R}^n)$ \\
$\GL(n)$ & $\P_1(k,n)$  & $\St_k(\mathbb{R}^n)$ \\
$\GL(n)$ & $\O(n)$  & $\El(\mathbb{R}^n)$
\end{tabu}
\end{table*}

The submanifolds $\V(k,n)$ and $\St(k,n)$ are respectively models for the compact and noncompact Stiefel manifolds given in Table~\ref{tab:sub}. The group of $2 \times 2$ block upper triangular matrices with invertible diagonal blocks,
\[
\P(k,n) \coloneqq \biggl\{ \begin{bmatrix}
X_1 & X_2 \\ 
0 & X_3
\end{bmatrix}
\in \GL(n) :
X_1\in \GL(k),\; X_2\in \mathbb{R}^{k\times (n-k)}, \; X_3\in \GL(n-k)
\biggr\}
\]
is  called a \emph{parabolic subgroup} of $\GL(n)$; and $\P_1(k,n)$ is the subgroup of $\P(k,n)$ with $X_1 = I_k$.

The last column in Table~\ref{tab:quo} gives the abstract manifold that $\mathcal{M} / G$ is diffeomorphic to but note that $\mathcal{M}/G$ inherits an extrinsic coordinate system given by $\mathcal{M}$ and the $G$-action that is absent in its abstract counterpart.  Some of the manifolds in Table~\ref{tab:quo} are well-known homogeneous spaces and some even symmetric spaces \cite{Helgason01} but we will not use these properties. The first two quotient models in Table~\ref{tab:quo} have formed the mainstay of Grassmannian manifold optimization with tens of thousands of works, too numerous for us to identify any handful of representative articles, and we will just mention the pioneering feat of \cite{EAS99}.

The submanifold models are simply matrix submanifolds $\mathcal{M}$ diffeomorphic to one of the abstract manifold in the last column in Table~\ref{tab:sub}, which lists all such models for the  Stiefel, Grassmann, and Cartan manifolds known to us:
\begin{table*}[h]
\small
\caption{Submanifold models}
\label{tab:sub}
\tabulinesep=0.75ex
\begin{tabu}{@{}lr@{$\null=\null$}ll}
model &  \multicolumn{2}{c}{concrete}  & abstract \\\tabucline[0.5pt ] -
positive definite model & $\mathbb{S}^n_\pp $ & $\{ A \in \mathbb{R}^{n \times n} : A^\tp = A, \; x^\tp A x > 0 \text{ for all } x \ne 0 \}$ & $\El(\mathbb{R}^n)$  \\
orthonormal model & $\V(k,n) $ & $\{Y \in \mathbb{R}^{n \times k} :  Y^\tp Y = I \}$ & $\V_k(\mathbb{R}^n)$ \\
full-rank model & $\St(k,n)$ & $\{ S \in \mathbb{R}^{n \times k} : \rank(S) = k\}$ & $\St_k(\mathbb{R}^n)$ \\
projection model & $\Gr_\pi(k, n)$ & $\{P \in \mathbb{R}^{n \times n}: P^2 = P = P^\tp,\; \tr(P) = k\}$ & $\Gr_k(\mathbb{R}^n)$ \\
involution model & $\Gr_\iota(k,n)$ & $\{Q \in \mathbb{R}^{n \times n} : Q^\tp Q =I,\; Q^\tp = Q, \; \tr(Q)=2k - n\}$ & $\Gr_k(\mathbb{R}^n)$  \\
quadratic model & $\Gr_{a,b}(k, n)$ & $\begin{multlined}[t] \{ W \in \mathbb{R}^{n \times n}: W^\tp = W, \; (W - a I)(W- bI) = 0, \\[-2ex]
 \tr(W) = ka + (n-k)b \} \end{multlined}$ & $\Gr_k(\mathbb{R}^n)$ 
\end{tabu}
\end{table*}

As in the case of the quotient models, these submanifold models endow the abstract manifold with an extrinsic coordinate system. The quadratic model for the $\Gr_k(\mathbb{R}^n)$ is an infinite family of submanifold models, one for each distinct pair of $a,b \in \mathbb{R}$, that has only recently been discovered  \cite[Theorem~6.1]{LY24}. It contains both the projection and involution models as special cases:
\[
\Gr_{1,0}(k, n) = \Gr_\pi(k,n), \qquad \Gr_{1,-1}(k, n) = \Gr_\iota(k,n).
\]
The projection model has a long history in applications including coding theory \cite{CHRSS,Conway}, machine learning \cite{EG}, optimization \cite{SSZ,Zhao}, statistics \cite{Chi}; and is also common in geometric measure theory \cite{Mattila} and differential geometry \cite{Nicolaescu}. The involution model is more recent \cite{ZLK20, ZLK24, LY24}.
We will show in Proposition~\ref{prop:redu} that any two different models in Tables~\ref{tab:quo} and \ref{tab:sub} for the same abstract manifold may be transformed to each other in polynomial-time.

By far the two most common models of the Grassmannian in pure mathematics are $\GL(n)/\P(k,n)$ in Table~\ref{tab:quo} and the \emph{Pl\"ucker model}. For the sake of completeness the latter is the image of the Pl\"ucker embedding $\Gr_k(\mathbb{R}^n) \to \mathbb{RP}^{\binom{n}{k}}$  \cite[Lecture~6]{Harris}. It is hard to imagine it ever finding practical use due to the exponential dimension of its ambient space. We will disregard it in the rest of this article.

\section{Background and assumptions}

We will discuss some background and assumptions necessary to put our complexity results on rigorous footing. In this article, ``unconstrained'' simply means that we do not impose any constraints on the optimization variables other than that the points must lie on the manifold in question. To formulate optimization problems that meet standard complexity theoretic assumptions like finite bit-length inputs (see Section~\ref{sec:cplx}), we limit our objective functions to polynomial functions with rational coefficients. Hence we will need to define what we mean by a polynomial function in the context of manifold optimization. This is straightforward for the submanifold models in Table~\ref{tab:sub} but takes some small amount of algebraic geometry \cite{Harris} and geometric invariant theory \cite{MFK94}  for the quotient models in Table~\ref{tab:quo} (see Section~\ref{sec:poly}). 

\subsection{Polynomial optimization on submanifolds and quotient manifolds}\label{sec:poly}

This section may be safely skipped for readers willing to accept on faith that it is perfectly fine to speak of a polynomial function on any of the manifolds appearing in Tables~\ref{tab:quo} and \ref{tab:sub}. In the following, we provide a careful definition of a polynomial function over these manifolds, which we are unable to find in the manifold optimization literature. This is unsurprising given that polynomials and manifolds are from different categories: polynomial functions belong with varieties in algebraic geometry; smooth functions belong with manifolds in differential geometry. In general they are incompatible but fortunately for us, every candidate in Tables~\ref{tab:quo} and \ref{tab:sub} is not only a smooth manifold but also a variety or the projection of one.

A subset $\mathcal{M} \subseteq \mathbb{R}^{n \times k}$ is a \emph{variety} \cite[p.~3]{Harris} if it is defined by equality constraints that are polynomials in the entries of the matrix variable $X = (x_{ij})$. In Table~\ref{tab:sub}, the matrix submanifolds $\V(k,n)$, $\Gr_\pi(k,n)$, $\Gr_\iota(k,n)$, and $\Gr_{a,b}(k,n)$ are clearly varieties since they are defined with polynomial equality constraints. Note that this also applies to $\O(n) = \V(n,n)$ as a special case.

The remaining cases $\St(k,n)$ and $\mathbb{S}^n_\pp $ are more subtle. Let $\pi_1 : S_1 \times S_2 \to S_1$, $\pi_1(s_1, s_2) = s_1$, be the projection map onto the first factor in a product of two sets. Recall that a matrix $X$ has full column rank iff $X^\tp X$ is nonsingular; and so we have
\begin{equation}\label{eq:proj1}
\St(k,n) = \{ X \in \mathbb{R}^{n \times k} : \det(X^\tp X) \ne 0 \}  = \pi_1 \bigl( \{ (X, t) \in \mathbb{R}^{n \times k} \times \mathbb{R} : t \det(X^\tp X) = 1 \} \bigr).
\end{equation}
Similarly, an $n \times n$ matrix is positive definite iff all its $m \coloneqq 2^n-1$ principal minors are positive; and so we have
\begin{align*}
\mathbb{S}^n_\pp  &= \{ X \in \mathbb{R}^{n \times n} : X^\tp = X, \; \Delta_i(X) > 0,\; i = 1,\dots, m\} \\
&= \pi_1 \bigl( \{ (X, t) \in \mathbb{R}^{n \times n} \times \mathbb{R}^m : X^\tp = X, \; t_i^2 \Delta_i(X)=1,\; i = 1,\dots, m\} \bigr).
\end{align*}
Here $\Delta_1(X),\dots,\Delta_m(X)$ are the principal minors of $X$. In optimization theoretic lingo, we have introduced slack variables to convert ``$\ne$'' and ``$>$'' into ``$=$'' and express $\St(k,n)$ and  $\mathbb{S}^n_\pp $ in terms of polynomial equality constraints. In algebraic geometric lingo, $\St(k,n)$ and  $\mathbb{S}^n_\pp $ are projections of varieties in higher-dimensional ambient spaces. The slack variables supply the extra dimensions.

In fact, we can say more.  In \eqref{eq:proj1}, we actually have an isomorphism
\[
\{ (X, t) \in \mathbb{R}^{n \times k} \times \mathbb{R} : t \det(X^\tp X) = 1 \} \cong \{ X \in \mathbb{R}^{n \times k} : \det(X^\tp X) \ne 0 \}
\]
since the map $(X,t) \mapsto X$ is invertible with inverse $X \mapsto (X,1/\det(X^\tp X))$. In other words, $\St(k,n)$ can be turned into a variety with an appropriate change of cooordinates. Note that everything we said about $\St(k,n)$ applies to the special case $\GL(n) =\St(n,n)$ too. Now observe that $\mathbb{S}^n_\pp $ is a connected component of $\GL(n) \cap \mathbb{S}^n$, which is isomorphic to a variety by the same argument.

It is straightforward to define polynomial functions for the submanifold models in Table~\ref{tab:sub}:
\begin{definition}\label{def:polysub}
Let $n,k \in \mathbb{N}$. A polynomial function $f : \mathbb{R}^{n \times k} \to \mathbb{R}$ is a function defined by a polynomial in $\mathbb{R}[x_{11}, x_{12},\dots,x_{nk}]$, i.e., a polynomial in the entries of the $n \times k$ matrix variable $X = (x_{ij})$. The degree of $f$ is simply its degree as a polynomial. A \emph{polynomial function} on any subset $\mathcal{M} \subseteq \mathbb{R}^{n \times k}$ is a polynomial function on $ \mathbb{R}^{n \times k}$ restricted to $\mathcal{M}$.
\end{definition}

Definition~\ref{def:polysub} applies to any subset $\mathcal{M} \subseteq \mathbb{R}^{n\times k}$ but if $\mathcal{M}$ is a variety,  then the set of all polynomial functions on $\mathcal{M}$ has a special structure \cite[p.~18]{Harris} and is called the \emph{coordinate ring} of $\mathcal{M}$, denoted $\mathbb{R}[\mathcal{M}]$. For $\mathcal{M} = \St(k,n)$ or $\GL(n) \cap \mathbb{S}^n$, which are only isomorphic to varieties but not themselves varieties, the coordinate ring $\mathbb{R}[\mathcal{M}]$ is the set of all \emph{regular functions}  \cite[p.~19]{Harris}  on $\mathcal{M}$, which include all polynomial functions but also some nonpolynomial functions. Since $\mathbb{S}^n_\pp  \subseteq \GL(n) \cap \mathbb{S}^n$,  a polynomial (resp.\ regular) function on $\mathbb{S}^n_\pp$ is simply a polynomial (resp.\ regular) function on $\GL(n) \cap \mathbb{S}^n$ restricted to $\mathbb{S}^n_\pp$. We may only  speak of the coordinate ring  $\mathbb{R}[\mathcal{M}]$ when $\mathcal{M}$ is a variety or is isomorphic to one, which is why $\mathbb{S}^n_\pp $ has to be treated separately.

The exact definitions of coordinate ring and regular function  \cite[pp.~18--19]{Harris} are unimportant for us. All the reader needs to know for the discussion below is that if $\mathcal{M}$ is any matrix submanifold appearing in the first column of Table~\ref{tab:quo}, then its coordinate ring $\mathbb{R}[\mathcal{M}]$ contains all polynomial functions on $\mathcal{M}$ but may also contain additional regular functions that are nonpolynomial.

We will define a polynomial function on any quotient manifold in Table~\ref{tab:quo} by reducing to Definition~\ref{def:polysub} by way of a well-known result in geometric invariant theory \cite[Amplification~1.3]{MFK94}: If all $G$-orbits are closed,  then the orbit space $\mathcal{M}/G$ is isomorphic to a variety and
\begin{equation}\label{eq:iso}
\mathbb{R}[\mathcal{M}/G] \simeq \mathbb{R}[\mathcal{M}]^G.
\end{equation}
Here  $\mathbb{R}[\mathcal{M}]^G$ denotes the subring of $G$-invariant regular functions in the coordinate ring $\mathbb{R}[\mathcal{M}]$.  The quotient manifolds in Table~\ref{tab:quo} are all orbit spaces of the form $\mathcal{M}/G $ with $\mathcal{M}$ a matrix submanifold listed in the first column and $G$ the corresponding matrix group listed in the second.  In all these cases,  we may verify that $G$-orbits in $\mathcal{M}$ are closed: This is obvious whenever $G$ is a compact group. The three cases involving noncompact $G$ require separate considerations. Take the  quotient model for the noncompact Stiefel manifold $\GL(n)/\P_1(k,n)$ for illustration: Here $\mathcal{M} = \GL(n)$, $G = \P_1(k,n)$, and the $\P_1(k,n)$-orbit of $\begin{bsmallmatrix}
X & Y \\
Z & W
\end{bsmallmatrix} \in \GL(n)$ is
\[
\biggl\lbrace
\begin{bmatrix}
X & Y' \\
Z & W'
\end{bmatrix} \in \GL(n): Y'\in \mathbb{R}^{k \times (n-k)},\; W'\in \mathbb{R}^{(n-k) \times (n-k)}
\biggr\rbrace,
\]
which is closed in $\GL(n)$. We may check that the same holds for $\mathcal{M} = \GL(n)$, $G = \P(k,n)$, and for  $\mathcal{M} = \St(k,n)$, $G = \GL(k)$. By \eqref{eq:iso},  an element of  $\mathbb{R}[\mathcal{M}/G]$, i.e., a regular function $f : \mathcal{M}/G \to \mathbb{R}$, corresponds uniquely to an element of $\mathbb{R}[\mathcal{M}]^G$, i.e., a $G$-invariant regular function on $\mathcal{M}$. With this, we shall define polynomial functions on $\mathcal{M}/G$ to be exactly those  elements of $\mathbb{R}[\mathcal{M}/G]$ that correspond to polynomial functions in $\mathbb{R}[\mathcal{M}]^G$.
\begin{definition}\label{def:polyquo}
Let $\mathcal{M}/G$ be any quotient manifold appearing in Table~\ref{tab:quo}.  A \emph{polynomial function} $f : \mathcal{M}/G  \to \mathbb{R}$ is a function defined by a $G$-invariant polynomial function on $\mathcal{M}$.
\end{definition}
In case the reader is wondering, the last row of Table~\ref{tab:quo} says that $\GL(n)/\O(n)$ is a quotient model for $\El(\mathbb{R}^n)$ and the first row of Table~\ref{tab:sub} says that $\mathbb{S}^n_\pp$ is a submanifold model for $\El(\mathbb{R}^n)$. The consequence is that $\GL(n)/\O(n)$ and $\mathbb{S}^n_\pp$ are diffeomorphic as manifolds. There is no contradiction with what the above discussions imply: $\GL(n)/\O(n)$ is isomorphic to a variety but $\mathbb{S}^n_\pp$ is not, as isomorphism of varieties is entirely different from diffeomorphism of manifolds.

Just so that the reader sees why we need to rely on \eqref{eq:iso} to define a polynomial: Take $\mathcal{M} = \GL(n)$ and  $G = \P_1(k,n)$ again. We cannot directly say what is a polynomial function $f : \GL(n)/\P_1(k) \to \mathbb{R}$ since a point in $\GL(n)/\P_1(k)$ is an equivalence class with no preferred representative, i.e., we cannot define $f$ by writing down a polynomial expression in the coordinates of the point.  However, with
\[
\mathbb{R}[\GL(n)/\P_1(k,n)] \simeq \mathbb{R}[\GL(n)]^{\P_1(k,n)},
\]
we know that $f$ corresponds uniquely to a regular function $\bar{f} : \GL(n) \to \mathbb{R}$. So if
\begin{enumerate}[\upshape (i)]
\item $\bar{f}$ is a polynomial function on $\GL(n)$ according to Definition~\ref{def:polysub};
\item $\bar{f}$ is $\P_1(k,n)$-invariant, i.e., $f(XP) = f(X)$ for all $X \in \GL(n)$ and $P \in \P_1(k,n)$;
\end{enumerate}
then we say that $f$ is a polynomial function on $\GL(n)/\P_1(k)$.

Unlike Definition~\ref{def:polysub}, we did not include a notion of degree in Definition~\ref{def:polyquo} as a rigorous definition is beyond the scope of this article.  To see why it can be tricky to define ``degree'' over a quotient model, take $\O(3)/(\O(2)\times \O(1))$ for illustration. We have
\begin{equation}\label{eq:cs}
\left\lb \begin{bmatrix}
c & 0 & -s \\
0 & 1 & 0 \\
s & 0 & c
\end{bmatrix} \right\rb = \left\lb
\begin{bmatrix}
c^3 - cs^2 & -2c^2s & -s \\
2cs & c^2 - s^2 & 0 \\
c^2s - s^3 & -2cs^2 & c
\end{bmatrix}\right\rb
\end{equation}
where $\lb \, \cdot \, \rb$ denotes $\O(2)\times \O(1)$-cosets in $\O(3)$. In fact the matrix representative on the right can be chosen to have arbitrary high powers in $c$ and $s$, noting that it is really just the $d =2$ case of
\[
\begin{bmatrix}
c & 0 & -s \\
0 & 1 & 0 \\
s & 0 & c
\end{bmatrix}
\begin{bmatrix}
c & -s & 0 \\
s & c & 0 \\
0 & 0 & 1
\end{bmatrix}^d.
\]
It is easy to declare that we can always just pick a lowest-degree representative, but as the matrix representive on the right of \eqref{eq:cs} shows, even in an example as simple as this one, it is rarely clear that the given representative can be further simplified.
Of course with a more careful treatment we may circumvent this issue and still define degree rigorously for a quotient model. However we do not wish to distract the reader with such tangential matters.
The small price we pay is that we may only speak of ``polynomial optimization'' but not ``linear, quadratic, cubic, or quartic'' programming when discussing quotient models.

\subsection{Complexity theoretic assumptions}\label{sec:cplx}

Just as we need a model for an abstract manifold to discuss manifold optimization, we need a model of complexity to discuss NP-hardness --- Cook--Karp--Levin, Blum--Shub--Smale, Valiant, or yet something else. We will use  that of Cook--Karp--Levin, more commonly known as the Turing \emph{bit model} in optimization and widely adopted by optimization theorists; see \cite[Section~1.1]{GLS} and \cite[Chapter~2]{vava}. But on occasion we will instead use the \emph{oracle model} \cite[Section~1.2]{GLS}, also called black-box model in \cite[Chapter~6]{vava}. This is a dichotomy that we will need to highlight:
\begin{enumerate}[\upshape (a)]
\item\label{it:poly} In polynomial optimization, the input consists of the polynomial's coefficients, assumed rational, and thus bit complexity is well-defined.

\item\label{it:cvx} In convex optimization, the input is a convex function; and in general there is no way to quantify how many bits it takes to describe an arbitrary  convex function. The standard approach around this conundrum is to use oracle complexity  \cite{NY1983, nesterov2018}.
\end{enumerate}
Unless specified otherwise all of our NP-hardness results will be in terms of the more precise bit complexity \eqref{it:poly}. In a handful of less important corollaries when a bit complexity version is out of reach we resort to the oracle complexity model \eqref{it:cvx} --- these will be clearly specified.

We will also reproduce \cite[Definition~2.5]{de2008} for FPTAS below for easy reference.
\begin{definition}[Fully polynomial-time approximation scheme]\label{def:FPTAS}
With respect to the maximization problem over $\mathcal{M}$ and the function class $\mathscr{F}$, an algorithm $\mathscr{A}$ is called a fully polynomial-time approximation scheme or $\mathrm{FPTAS}$ if:
\begin{enumerate}[\upshape (i)]
\item For any instance $f \in \mathscr{F}$ and any $\varepsilon>0$, $\mathscr{A}$ takes the defining parameters of $f$ (e.g., coefficients of $f$ when $f$ is a polynomial), $\varepsilon$, and $\mathcal{M}$ as input and computes an $x_\varepsilon \in \mathcal{M}$ such that $f(x_\varepsilon)$ is a $(1-\varepsilon)$-approximation of $f$, i.e.,
\[
f_{\max} - f(x_\varepsilon)  \le  \varepsilon(f_{\max} - f_{\min}).
\]

\item The number of operations required for the computation of $x_\varepsilon$ is bounded by a polynomial in the problem size, and $1/\varepsilon$.
\end{enumerate}
If $\mathscr{A}$ is an $\mathrm{FPTAS}$ that relies on an oracle of $f$ for the value of $f$ at some point, then we call it an oracle-$\mathrm{FPTAS}$.
\end{definition}

\section{NP-hardness of QP over $\Gr(k,n)$ without fixing $k$ and $n$}\label{sec:both}

In this and the next two sections, we will use the projection model  $\Gr_\pi(k, n)$ in Table~\ref{tab:sub} to model our Grassmannian. For notational simplicity we drop the subscript $\pi$ and just write
\begin{equation}\label{eq:iden}
\Gr(k, n) = \{P \in \mathbb{R}^{n \times n}: P^2 = P = P^\tp,\; \tr(P) = k\}.
\end{equation}
Observe that $\tr(P) =\rank(P)$ for an orthogonal projection matrix $P$.
Given that the Grassmannian parameterizes $k$-planes in $n$-space, one might expect to find a relation with the combinatorial problem of choosing a $k$-element set from an $n$-element set. This intuition pans out --- there is a natural correspondence between the two.

Let $G = (V, E)$ be an $n$-vertex undirected graph with $V = \{1,\dots,n\}$. We denote an edge as an ordered pair $(i,j)$. For an undirected graph, this means that if $(i,j) \in E$, then $(j,i) \in E$, and so a sum over $E$ sums each edge twice. We do not allow self loop so if $(i,j) \in E$, then $i \ne j$.

The clique decision problem asks if a clique of size $k$ exists in $G$. The problem is one of Karp's original 21 NP-complete problems \cite{Karp}.  Let $e_1, \dots, e_n  \in \mathbb{R}^n$ be its standard basis. Consider the following maximization problem
\begin{equation}\label{eq:func}
\max_{P\in\Gr(k, n)} f(P) \coloneqq  \max_{P\in\Gr(k, n)}  \biggl[  \sum_{(i, j) \in E}e_i^\tp P e_i e_j^\tp P e_j + \sum_{i \in V} e_i^\tp P e_i e_i^\tp P e_i \biggr].
\end{equation}
As a manifold optimization problem over $\Gr(k,n)$, the problem is unconstrained, and clearly quadratic, in fact homogeneous of degree two as $f(c P) = c^2 f(P)$ for any $c \in \mathbb{R}$.

\begin{proposition}[Clique decision as Grassmannian QP]\label{prop:cli}
Any $n$-vertex graph $G$ contains a $k$-clique if and only if $\max_{P\in\Gr(k, n)} f(P) = k^2$.
\end{proposition}
\begin{proof}
Any projection matrix $P = (p_{ij})  \in \Gr(k,n)$ corresponds to the unique $k$-dimensional subspace $\im(P) \subseteq \mathbb{R}^n$. Observe that $p_{ii} = e_i^\tp P e_i = 1$ if and only if $e_i \in \im(P)$. 
If $G$ contains a $k$-clique, we may choose $P$ to be the diagonal matrix $D$ with $d_{ii} = 1$ when $i$ is a vertex in the $k$-clique and zero everywhere else, then $f(D) = k^2$. As $P \in \Gr(k,n)$ is automatically positive semidefinite, we always have $p_{ii} \ge 0$ for all $ i=1,\dots, n$. This implies
\[
f(P) =  \sum_{(i, j) \in E} p_{ii}p_{jj} + \sum_{i \in V} p_{ii}^2 \le  \sum_{i \ne j} p_{ii}p_{jj} + \sum_{i=1}^n p_{ii}^2 = \tr(P)^2 = k^2.
\]
So when there is a $k$-clique, we have $\max_{P\in\Gr(k, n)} f(P) = k^2$. For any $P\in\Gr(k, n)$, we may also write
\[
f(P) =  \sum_{i \ne j} p_{ii}p_{jj} + \sum_{i \in V} p_{ii}^2 -  \sum_{(i, j)\notin E}p_{ii}p_{jj} = k^2 - \sum_{(i, j)\notin E}p_{ii}p_{jj}.
\]
The diagonal entries satisfy $p_{11} + \dots + p_{nn} = \tr(P) = k$ and $0 \le p_{11},\dots,p_{nn} \le 1$. We may assume that $1 \ge  p_{11} \ge  p_{22} \ge  \dots \ge  p_{nn} \ge  0$. Then
\[
k = \sum_{i = 1}^n p_{ii} \le (k-1) + (n-k+1)p_{kk},
\]
so $p_{kk} \ge  1/(n-k-1)$. If $G$ contains no $k$-clique, then in particular the vertices $1, \dots, k$ do not form a clique. So $(i_0, j_0) \notin E$ for some $i_0,j_0 \in \{1,\dots, k \}$. Hence $p_{i_0 i_0},  p_{j_0 j_0} \ge p_{kk}$,
\[
\sum_{(i, j)\notin E}p_{ii}p_{jj} \ge  p_{kk}^2 \ge \frac{1}{(n-k-1)^2},
\]
and
\[
f(P) \le k^2 - \frac{1}{(n-k-1)^2},
\]
as required.
\end{proof}

Applying Definition~\ref{def:FPTAS} to the current scenario, the input data for $\mathcal{M} = \Gr(k,n)$ is just $k$ and $n$; and the input data for an instance $f \in \mathscr{F}$, the function class in \eqref{eq:func}, is just a graph $G = (V,E)$.  We may deduce the following hardness result from Proposition~\ref{prop:cli}.
\begin{theorem}[Grassmannian quadratic programming is NP-hard I]
Unless $\mathrm{P} = \mathrm{NP}$, there is no $\mathrm{FPTAS}$ that is polynomial in $n$ and $k$ for maximizing quadratic polynomials over $\Gr(k, n)$.
\end{theorem}
\begin{proof}
If there is a $k$-clique in $G$, then $f_{\max} = k^2$ by Proposition~\ref{prop:cli}. Take
\[
\varepsilon = \frac{1}{2k^2(n-k-1)^2},
\]
so that $1/\varepsilon$ is polynomial in $k, n$. Then an $\mathrm{FPTAS}$ gives a $P_\varepsilon \in \Gr(k,n)$ such that
\[
f(P_\varepsilon) \ge  f_{\max} - \varepsilon(f_{\max} - f_{\min}) \ge  k^2 - \frac{1}{2k^2(n-k-1)^2}k^2 = k^2 - \frac{1}{2(n-k-1)^2}.
\]
If there is no $k$-clique in $G$, then an $\mathrm{FPTAS}$ algorithm gives a $P_\varepsilon \in \Gr(k,n)$ such that
\[
f(P_\varepsilon) \le  f_{\max} \le  k^2 - \frac{1}{(n-k-1)^2}.
\]
In other words, we can decide existence or non-existence of a $k$-clique in $G$ in polynomial time if there were an $\mathrm{FPTAS}$ for the maximization problem.
\end{proof}

\section{NP-hardness of QP over $\Gr(k,n)$ with fixed $k$}\label{sec:onlyn}

We remind the reader that NP-hardness is a notion of \emph{asymptotic} time complexity. In the last section, we assume that $k$ and $n$ both grow to infinity. In this section we will prove a stronger variant allowing $k \in \mathbb{N}$ to be arbitrary but fixed and only $n$ grows to infinity.

We let $G = (V, E)$ be an $n$-vertex undirected graph following conventions set in the last section. For any $n, k \in \mathbb{N}$ with  $k \le n$, we define the set
\[
\Delta_{k, n} \coloneqq \{x \in \mathbb{R}^n : 0\le  x_i \le  1, \; i=1, \dots, n,\; x_1 + \dots + x_n = k\},
\]
and the function $f : \mathbb{R}^n \to \mathbb{R}$,
\[
f(x) \coloneqq  \sum_{(i, j) \in E} x_i x_j.
\]

Recall that $S\subseteq V$ is a clique if $(i, j) \in E$ for all $i, j \in S$. We will show that the \emph{clique number}
\[
\omega(G) \coloneqq \max\{ \lvert S \rvert : S \subseteq V \text{ is a clique}\}
\]
of $G$ may be determined by maximizing $f$ over $\Delta_{k,n}$. Indeed, for the special case $k =1$,
\[
\Delta_{1, n} = \{x \in \mathbb{R}^n :  x_1 + \dots + x_n = 1, \; x_i \ge 0, \; i=1, \dots, n\}
\]
is just the unit simplex and we have the celebrated result of Motzkin and Straus \cite{motzkin1965} that
\begin{equation}\label{eq:MS}
\max_{x \in \Delta_{1, n}}  f(x) = 1-\frac{1}{\omega(G)}.
\end{equation}
Note that because of our convention of summing over each undirected edge twice, our expression is slightly neater, lacking the factor $1/2$ on the right found in \cite{motzkin1965}.
We will extend this to $\Delta_{k,n}$.
\begin{proposition}[Generalized Motzkin--Straus]\label{prop:MS}
Let $k \le n$ be positive integers and $G$ be an $n$-vertex undirected graph. If $\omega(G) \ge  k$, then
\[
\max_{x \in \Delta_{k, n}} f(x) = k^2\biggl(1-\frac{1}{\omega(G)}\biggr).
\]
\end{proposition}
\begin{proof} 
Suppose $S = \{1, \dots, \omega(G)\} \subseteq V$ is a largest clique. Then setting $x_* \in \Delta_{k, n}$ with coordinates
\[
x_1^* = \dots = x_{\omega(G)}^* = \frac{k}{\omega(G)}, \quad x_{\omega(G) + 1}^* = \dots = x_n^* = 0
\]
gives us the value
\[
f(x_*) = k^2\biggl(1-\frac{1}{\omega(G)}\biggr).
\]
So we only need to show that
\[
\max_{x \in \Delta_{k, n}} f(x) \le  k^2\biggl(1-\frac{1}{\omega(G)}\biggr).
\]
By a change-of-variable $x = ky$, the maximization problem becomes
\[
\max_{y \in \Theta_{k,n}} k^2f(y).
\]
where
\[
\Theta_{k,n} \coloneqq \biggl\{y \in \mathbb{R}^n :  y_1 + \dots + y_n = 1, \; 0 \le  y_i \le  \frac{1}{k}, \; i=1, \dots, n\biggr\} \subseteq \Delta_{1, n}.
\]
So by the original Motzkin--Straus Theorem \eqref{eq:MS},
\[
\max_{x \in \Delta_{k, n}} f(x) = k^2 \max_{y \in \Theta_{k,n}} f(y) \le  k^2 \max_{y \in \Delta_{1, n}} f(y) \le  k^2\biggl(1-\frac{1}{\omega(G)}\biggr),
\]
as required.
\end{proof}
While we will not need this, but the statement in Proposition~\ref{prop:MS} is really ``if and only if.'' We may show that if $\omega(G) < k$, then
\[
\max_{x \in \Delta_{k, n}} f(x) < k^2\biggl(1-\frac{1}{\omega(G)}\biggr).
\]
In this sense the result in Proposition~\ref{prop:MS} is sharp.

\begin{proposition}[Clique number as Grassmannian QP]\label{prop:cli2}
Let $k \le n$ be positive integers and $G$ be an $n$-vertex undirected graph. 
If $\omega(G) \ge  k$, then
\[
\max_{P \in \Gr(k, n)} \sum_{(i, j) \in E} p_{ii} p_{jj} = k^2\biggl(1-\frac{1}{\omega(G)}\biggr).
\]
\end{proposition}
\begin{proof}
Here $\Gr(k,n)$ is as in \eqref{eq:iden}. By the Schur--Horn Theorem \cite{horn1954}, the diagonal map $\diag : \mathbb{R}^{n \times n} \to \mathbb{R}^n$, when restricted to $\Gr(k,n)$, gives a surjection
\[
\diag : \Gr(k,n) \to  \Delta_{k,n}, \quad P \mapsto \diag(P).
\]
In other words, with a relabeling of variables $x_i = p_{ii}$, $i=1,\dots,n$, we get
\[
\max_{P \in \Gr(k, n)} \sum_{(i, j) \in E} p_{ii} p_{jj}  = \max_{x \in \Delta_{k, n}}  \sum_{(i, j) \in E} x_i x_j.
\]
So the required result follows from Proposition~\ref{prop:MS}. As in the previous section, the objective function may be written $\sum_{(i, j) \in E}  e_i^\tp P e_i e_j^\tp P e_j$, which is clearly quadratic in $P$.
\end{proof}

Note that our next proof will rely on the fact that for any fixed $k$, the time complexity of the $k$-clique problem  is polynomial in $n$. Even an exhaustive search through all subgraphs with $k$ or fewer vertices has time $\binom{n}{1} + \binom{n}{2} + \dots + \binom{n}{k}$, a degree-$k$ polynomial in $n$.
\begin{theorem}[Grassmannian quadratic programming is NP-hard II]\label{thm:NP2}
Let $k \in \mathbb{N}$ be fixed. Unless $\mathrm{P} = \mathrm{NP}$, there is no $\mathrm{FPTAS}$ that is polynomial in $n$ for maximizing quadratic polynomials over $\Gr(k, n)$.
\end{theorem}
\begin{proof}
Given any $n$-vertex graph $G =(V,E)$ as input, we may  test all subgraphs with $k$ or fewer vertices and check if $\omega(G) <  k$ in time polynomial in $n$. If  $\omega(G) < k$, this also yields the exact value of $\omega(G)$.
If  $\omega(G) \ge k$, then we find an $\varepsilon$-approximate solution to
\[
\max_{P \in \Gr(k, n)} \sum_{(i, j) \in E} p_{ii} p_{jj}
\]
using an $\mathrm{FPTAS}$, assuming one exists, where
\[
\varepsilon = \frac{1}{2k^2n^2}.
\]
The maximum only takes finitely many different discrete values depending on what $\omega(G)$ is, and these values have relative gaps of at least $\varepsilon$.  So if the maximum value can be approximated to $\varepsilon$-accuracy, then by Proposition~\ref{prop:MS} we get the exact value of $\omega(G)$ . Hence unless $\mathrm{P} = \mathrm{NP}$, there is no $\mathrm{FPTAS}$ for maximizing quadratic polynomials over $\Gr(k, n)$.
\end{proof}

Since orthogonal projection matrices are always positive semidefinite, we have
\[
\Gr(1,n) = \{ P \in \mathbb{S}^n_\p : P^2 = P, \; \tr(P) = 1\},
\]
i.e., it is exactly the set of density matrices of pure states. Its convex hull is then the set of all \emph{density matrices} of mixed states \cite[Theorem~2.5]{QC},
\[
\mathscr{D}_n \coloneqq \{ X \in \mathbb{S}^n_\p : \tr(X) = 1 \} = \conv \Gr(1,n).
\]
From this we obtained the following slightly unexpected payoff.
\begin{corollary}[Quadratic programming over density matrices is NP-hard]\label{coro:density}
Unless $\mathrm{P} = \mathrm{NP}$, there is no $\mathrm{FPTAS}$ that is polynomial in $n$ for maximizing quadratic polynomials over $\mathscr{D}_n$.
\end{corollary}
\begin{proof}
Following the proof of Proposition~\ref{prop:cli2}, the diagonal map $\diag : \mathbb{R}^{n \times n} \to \mathbb{R}^n$, when restricted to $\mathscr{D}_n$, gives a surjection
\[
\diag : \mathscr{D}_n \to  \Delta_{1,n}, \quad X \mapsto \diag(X).
\]
Relabeling $x_i = x_{ii}$, $i=1,\dots,n$, we get
\[
\max_{X \in \mathscr{D}_n} \sum_{(i, j) \in E} x_{ii} x_{jj}  = \max_{x \in \Delta_{1, n}}  \sum_{(i, j) \in E} x_i x_j = 1-\frac{1}{\omega(G)}
\]
by \eqref{eq:MS}. Now it remains to repeat, with $k =1$, the same argument in the second paragraph of the proof of Theorem~\ref{thm:NP2}.
\end{proof}
Observe that we would not have been able to obtain this using the results in Section~\ref{sec:both}, which do not allow $k$ to be fixed.

\section{NP-hardness for other models of the Grassmannian}\label{sec:other}

We give a list of explicit, polynomial-time change-of-coordinates formulas for transformation between various models of the Grassmannian, from which it will follow that our NP-hardness results in Sections~\ref{sec:both} and \ref{sec:onlyn} apply to all models of the Grassmannian in  Tables~\ref{tab:quo} and \ref{tab:sub}.

\begin{proposition}\label{prop:redu}
There exist maps $\varphi_1,\dots, \varphi_5$ in the diagram below
\begin{equation}\label{eq:cd}
\begin{tikzcd}
	{\O(n)/(\O(n-k)\times \O(k))} & {\V(k,n)/\O(k)} & {\Gr_\pi(k, n)} & {\Gr_{a,b}(k, n)} \\
	{\GL(n)/\P(k,n)} & {\St(k,n)/\GL(k)}
	\arrow["{\varphi_1}", from=1-1, to=1-2]
	\arrow["{{\varphi_5}}"', from=1-1, to=2-1]
	\arrow["{\varphi_2}", from=1-2, to=1-3]
	\arrow["{{\varphi_4}}"', from=1-2, to=2-2]
	\arrow["{\varphi_3}", from=1-3, to=1-4]
\end{tikzcd}
\end{equation}
that are diffeomorphisms of smooth manifolds, and polynomial-time computable.
\end{proposition}
\begin{proof}
We will give each of these maps and its inverse explicitly. As there are multiple quotient manifolds involved in \eqref{eq:cd}, we distinguish them by labeling with self-explanatory identifiers in subscript.

The map $\varphi_1 : \O(n)/(\O(n-k)\times \O(k)) \to \V(k,n) /\O(k)$ and its inverse are given by
\[
\varphi_1 (\lb Q \rb_{\O}) = \biggl\lb Q \begin{bmatrix} I_k \\ 0 \end{bmatrix} \biggr\rb_{\V},\quad \varphi_1^{-1} (\lb Y \rb_{\V}) = \lb Q_Y \rb_{\O}.
\]
Here $Q_Y \in \O(n)$ is the eigenvector matrix in the eigenvalue decomposition $YY^\tp = Q_Y \begin{bsmallmatrix}
I_k & 0 \\
0 & 0 
\end{bsmallmatrix} Q_Y^\tp$. Note that multiplying an $n \times n$ matrix on the right by $\begin{bsmallmatrix} I_k \\ 0 \end{bsmallmatrix} \in \mathbb{R}^{n \times k}$ gives an $n \times k$ matrix comprising the first $k$ columns. We check that
\begin{align*}
\varphi_1^{-1} \circ\varphi_1 (\lb Q \rb_{\O}) &= \varphi_1^{-1} \left(
\biggl\lb Q \begin{bmatrix} I_k \\ 0 \end{bmatrix} \biggr\rb_{\V}
\right) = \lb Q \rb_{\O},  \\
\varphi_1 \circ\varphi_1^{-1} (\lb Y \rb_{\V}) &= \varphi_1 (\lb Q_Y \rb_{\O}) = \biggl\lb Q_Y \begin{bmatrix} I_k \\ 0 \end{bmatrix} \biggr\rb_{\V} = \lb Y \rb_{V}.
\end{align*}
Here the last equality follows from the observation that $ZZ^\tp = Y Y^\tp$ for $Z, Y\in \V(k,n)$ if and only if $Z = Y P$ for some $P\in \O(k)$.

The map $\varphi_2 : \V(k,n) /\O(k) \to \Gr_\pi(k,n)$ and its inverse are given by
\[
\varphi_2 (\lb Y \rb_{\V}) = Y Y^\tp,\quad \varphi_2^{-1} (P) =\biggl\lb Q_P \begin{bmatrix} I_k \\ 0 \end{bmatrix} \biggr\rb_{\V}.
\]
Here $Q_P \in \O(n)$ is the eigenvector matrix in the eigenvalue decomposition $P = Q_P \begin{bsmallmatrix}
I_k & 0 \\
0 & 0 
\end{bsmallmatrix} Q_P^\tp$.  We check that 
\[
\varphi_2^{-1} \circ\varphi_2 (\lb Y \rb_{\V}) = \varphi_2^{-1} ( Y  Y^\tp) = \lb Y \rb_{\V},
\]
since the  column vectors of $Y$ are exactly the nonzero eigenvectors of the orthogonal projection matrix $Y Y^\tp$.  We also check that
\[
\varphi_2 \circ\varphi_2^{-1} (P) = \varphi_2 \biggl( \biggl\lb Q_P \begin{bmatrix} I_k \\ 0 \end{bmatrix} \biggr\rb_{\V}\biggr)  = Q_P \begin{bmatrix}
I_k & 0 \\
0 & 0 
\end{bmatrix} Q_P^\tp = P.
\]

The map $\varphi_3 : \Gr_\pi(k,n) \to \Gr_{a,b}(k,n)$ and its inverse are given by
\[
\varphi_3 (P) = (a - b) P + b I_n,  \quad \varphi_3^{-1} (W) = \frac{1}{a-b} (W - b I_n),
\]
which is obvious by inspection.

The map $\varphi_4 : \V(k,n) /\O(k) \to \St(k,n)/\GL(k)$ and its inverse are given by
\[
\varphi_4 ( \lb Y \rb_{\V}) = \lb Y \rb_{\St},\quad \varphi_4^{-1} ( \lb S \rb_{\St}) = \lb Y_S \rb_{\V}.
\]
Here $Y_S\in \V(k,n)$ is the unique orthogonal factor in the condensed QR decomposition $S = Y_S  R_S$ such that $R_S\in \mathbb{R}^{k\times k}$ is upper triangular with positive diagonal. We check that
\begin{align*}
\varphi_4^{-1} \circ \varphi_4 ( \lb Y \rb_{\V}) &= \varphi_4^{-1}(\lb Y \rb_{\St}) =  \lb Y \rb_{\V},  \\
\varphi_4 \circ \varphi_4^{-1} ( \lb S \rb_{\St}) &= \varphi_4(\lb Y_S \rb_{\V}) =  \lb Y_S \rb_{\St} = \lb Y_S R_S \rb_{\St} = \lb S \rb_{\St}.
\end{align*}

The map $\varphi_5 : \O(n)/(\O(n-k)\times \O(k)) \to \GL(n)/ \P(k,n)$ and its inverse are given by
\[
\varphi_5 ( \lb Q \rb_{\O}) = \lb Q \rb_{\GL},\quad \varphi_5^{-1} ( \lb X \rb_{\GL}) = \lb Q_X \rb_{\O}.
\]
Here $Q_X \in \O(n)$ is the unique orthogonal matrix in the full QR-decomposition $X = Q_X R_X$ where $R_X\in \mathbb{R}^{n\times n}$ is upper triangular with positive diagonal.  Since $R_X \in \P(k,n)$,  we have 
\begin{align*}
\varphi_5^{-1} \circ \varphi_5 ( \lb Q \rb_{\O}) &= \varphi_5^{-1}(\lb Q \rb_{\GL}) =  \lb Q \rb_{\O},  \\
\varphi_5 \circ \varphi_5^{-1} ( \lb X \rb_{\GL}) &= \varphi_1(\lb Q_X \rb_{\O}) =  \lb Q_X \rb_{\GL} = \lb Q_X R_X \rb_{\GL} =  \lb X \rb_{\GL}.
\end{align*}

Each $\varphi_i$ is smooth and its differential $d \varphi_i$ is an isomorphism at every point.  This is clear for $\varphi_3$ as it is an invertible linear map. This holds for $\varphi_2$ as it is induced by the smooth map $\overline{\varphi}_2: \V(k,n) \to \Gr_{\pi}(k,n)$, $Y \mapsto Y Y^\tp$; and $d_{\lb Y \rb} \varphi_2$ is an isomorphism as $d_Y \overline{\varphi}_2$ is surjective and its kernel is isomorphic to the Lie algebra $\mathfrak{o}(k)$ of $k \times k$ skew symmetric matrices.  The arguments for $\varphi_3$, $\varphi_4$, $\varphi_5$ are similar to that for $\varphi_2$.  Hence $\varphi_1,\dots, \varphi_5$ are all diffeomorphisms since they are also bijective \cite[Theorem~4.14]{Lee13}. 

Finally, since these diffeomorphisms $\varphi_1,\dots, \varphi_5$ and their inverses can be computed with matrix-matrix products, QR decomposition, and symmetric eigenvalue decomposition, all of which have $\mathcal{O}(n^\omega)$ time complexity,  with $\omega$ the exponent of matrix multiplication --- see \cite{Schonhage72} and \cite[Remark~16.26]{BCS97}. We conclude that these diffeomorphisms are all polynomial-time computable.
\end{proof}

By Proposition~\ref{prop:redu}, every arrow in \eqref{eq:cd} may be reversed, allowing one to get from any model to any other model in \eqref{eq:cd} by simply composing and inverting maps, with the resulting map remaining a polynomial-time computable diffeomorphism. But the reader is reminded of the dichotomy we highlighted on p.~\pageref{it:poly} --- having a polynomial-time computable diffeomorphism does not translate directly to polynomial-time reduction under the bit complexity model.

Take $\varphi_2 : \V(k,n) /\O(k) \to \Gr_\pi(k,n)$ for illustration. Saying that it is a  polynomial-time computable diffeomorphism implies that if we can evaluate  $f(P)$ for some $f : \Gr_\pi(k,n) \to \mathbb{R}$ and any $P \in \Gr_\pi(k,n)$ in a certain time complexity, then we can evaluate $f \circ \varphi_2(\lb Y \rb)$ for any $ \lb  Y \rb \in \V(k,n) /\O(k)$ with an increase in time complexity of at most a polynomial factor. It does not automatically imply that we have a polynomial-time algorithm to transform the defining parameters of $f$ into the defining parameters of $f \circ \varphi_2$. Example~\ref{eg:quad} illustrates this in the case when $f$ is a quadratic polynomial, showing how the defining parameters of $f$ are transformed into those of $f \circ \varphi_2$. Corollary~\ref{cor:poly}\eqref{it:para} shows when $f$ is any polynomial function, then we do indeed have a polynomial-time (in fact, constant-time) algorithm to transform the defining parameters of $f$ to those of $f \circ \varphi_2$ and $f \circ \varphi_2 \circ \varphi_1$.

\begin{example}[Quadratic case]\label{eg:quad}
Suppose $f : \Gr_\pi(k,n) \to \mathbb{R}$ is a quadratic polynomial, i.e.,
\begin{equation}\label{eq:quad}
	f(P) = \sum_{i,j,r,s=1}^n a_{ijrs} p_{ij} p_{rs} + \sum_{i,j=1}^n b_{ij} p_{ij} + c.
\end{equation}
We assume $c = 0$. Then on $\V(k,n)/\O(k)$,
\begin{multline}\label{eq:quadY}
f\circ \varphi_2(\lb Y \rb_{\V}) = f(YY^\tp)\\
=\sum_{i,j,r,s=1}^n a_{ijrs} (y_{i1}y_{j1} + \dots + y_{ik}y_{jk}) (y_{r1}y_{s1} + \dots + y_{rk}y_{sk}) + \sum_{i,j=1}^n b_{ij} (y_{i1}y_{j1} + \dots + y_{ik}y_{jk})
\end{multline}
where $Y \in \V(k,n)$ is any representative of $\lb Y \rb_{\V} \in \V(k, n)/\O(k)$. To be even more concrete, suppose $k = 2$, then the $(i,j,r,s)$th term in the first sum takes the form
\[
a_{ijrs} y_{i1}y_{j1}y_{r1}y_{s1} + a_{ijrs} y_{i2}y_{j2}y_{r1}y_{s1} + a_{ijrs}y_{i1}y_{j1} y_{r2}y_{s2} + a_{ijrs}y_{i2}y_{j2} y_{r2}y_{s2}
\]
and the $(i,j)$th term in the second sum takes the form
\[
 b_{ij} y_{i1}y_{j1} +  b_{ij}y_{i2}y_{j2}.
\]
In general, the coefficient $a_{ijrs}$ in $f$ will be repeated $k^2$ times in the defining parameters of $f\circ \varphi_2$  and the coefficient $b_{ijrs}$ will be repeated $k$ times. More generally, a degree-$d$ term in $f$ will see its coefficient repeated $k^d$ times as  defining parameters of $f\circ \varphi_2$. This discussion applies verbatim to $f \circ \varphi_2 \circ \varphi_1$. \qed
\end{example}

We show that Example~\ref{eg:quad} holds true more generally for any polynomial function $f$.
\begin{corollary}\label{cor:poly}
Let $ f:\Gr_\pi(k, n) \to \mathbb{R}$ be a polynomial function of bounded degree. Then $f$ can be transformed in polynomial-time into:
\begin{enumerate}[\upshape (a)]
\item a polynomial function on $\Gr_{a,b}(k, n)$ of the same degree as $f$ and whose coefficients are affine combinations of the coefficients of $f$;
\item\label{it:para} a polynomial function on $\O(n)/(\O(n-k)\times \O(k))$ or $\V(k,n)/\O(k)$ defined by parameters given by the same coefficients of $f$ but repeated with some multiplicities.
\end{enumerate}
\end{corollary}
\begin{proof}
From the proof of Proposition~\ref{prop:redu}, $\varphi_3^{-1} : \Gr_{a,b}(k,n) \to  \Gr_\pi(k,n)$ is an invertible affine map. So $f \circ \varphi_3^{-1}$ is a polynomial function of the same degree as $f$ whose coefficients are given by affine combinations of those in $f$.

For the map  $\varphi_1 : \O(n)/(\O(n-k)\times \O(k)) \to \V(k,n) /\O(k)$, observe that no matter which representative $Q \in \O(n)$ we pick for $\lb Q \rb_{\O}$, the matrix  $Q \begin{bsmallmatrix} I_k \\ 0 \end{bsmallmatrix} \in \V(k,n)$ will always be a valid representative for $\varphi_1(\lb Q \rb_{\O})$. For the map  $\varphi_2 : \V(k,n) /\O(k) \to  \Gr_\pi(k,n)$, observe that no matter which representative $Y \in \V(k,n)$ we pick for $\lb Y \rb_{\V}$, it will always be mapped to the same matrix  $YY^\tp \in \Gr_\pi(k,n)$. So these functions are well-defined, i.e., they give the same value irrespective of the representative picked.

The value of $f \circ \varphi_2$ on $\lb Y \rb_{\V} \in \V(k, n)/\O(k)$ is given by $f(YY^\tp)$ where $Y \in \V(k,n)$ is any representative of $\lb Y \rb_{\V}$. The value of $f \circ \varphi_2 \circ \varphi_1$ on $\lb Q \rb_{\O} \in \O(n)/(\O(n-k) \times \O(k))$ is given by $f\bigl(Q\begin{bsmallmatrix} I_k & 0 \\ 0 & 0 \end{bsmallmatrix} Q^\tp\bigr)$ where $Q \in \O(n)$ is any representative of $\lb Q \rb_{\O}$. In other words the parameters defining the map $f \circ \varphi_2$ or  $f \circ \varphi_2 \circ \varphi_1$ are just copies of the coefficients of $f$.
\end{proof}

We now state our main result of this section.
\begin{theorem}[Grassmannian optimization is NP-hard I]\label{thm:NP3}
Let $k \in \mathbb{N}$ be fixed. Unless $\mathrm{P} = \mathrm{NP}$, there is no $\mathrm{FPTAS}$ that is polynomial in $n$ for:
\begin{enumerate}[\upshape (i)]
\item maximizing quadratic polynomials over $\Gr_{a,b}(k, n)$;
\item maximizing polynomial functions over $\O(n)/(\O(n-k)\times \O(k))$ or $\V(k,n)/\O(k)$.
\end{enumerate}
\end{theorem}
\begin{proof}
If there were an  $\mathrm{FPTAS}$ for maximizing polynomial functions on $\V(k, n)/\O(k)$, we have an  $\mathrm{FPTAS}$ for maximizing any  quadratic polynomial $ f: \Gr_\pi(k,n) \to \mathbb{R}$ by maximizing
\[
f \circ \varphi_2 : \V(k, n)/\O(k) \to \mathbb{R},
\]
contradicting Theorem~\ref{thm:NP2}. The same argument applies to $f \circ \varphi_3^{-1}$ and $f \circ \varphi_2 \circ \varphi_1$ for the other two models.
\end{proof}
Note that Theorem~\ref{thm:NP3} is in terms of standard bit complexity.  But for the next result we will need to use the oracle complexity mentioned on p.~\pageref{it:cvx}.
\begin{corollary}[Grassmannian optimization is NP-hard II]\label{cor:St}
Let $k \in \mathbb{N}$ be fixed. Unless $\mathrm{P} = \mathrm{NP}$, there is no oracle-$\mathrm{FPTAS}$ that is polynomial in $n$ for maximizing polynomial functions over $\St(k, n)/\GL(k)$ or $\GL(n)/\P(k,n)$.
\end{corollary}
\begin{proof}
By Proposition~\ref{prop:redu},
\[
\varphi_2: \V(k,n) /\O(k) \to \Gr_\pi(k,n) , \quad \varphi_4^{-1}  :  \St(k,n)/\GL(k) \to  \V(k,n) /\O(k)
\]
are both polynomial-time computable diffeomorphisms. If there were an  $\mathrm{FPTAS}$ for maximizing polynomial functions on $\St(k, n)/\GL(k)$, we have an   $\mathrm{FPTAS}$ for maximizing any  quadratic polynomial $ f: \Gr_\pi(k,n) \to \mathbb{R}$ by querying
\[
f \circ \varphi_2\circ \varphi_4^{-1} : \St(k, n)/\GL(k) \to \mathbb{R}
\]
as an oracle, contradicting Theorem~\ref{thm:NP2}. The same argument also applies to
\[
f  \circ \varphi_2 \circ \varphi_1\circ \varphi_5^{-1} : \GL(n)/\P(k,n) \to \mathbb{R}. \qedhere
\]
\end{proof}
The proof of Corollary~\ref{cor:St} superficially resembles that of Corollary~\ref{cor:poly}; the difference being that $\varphi_1$ and $\varphi_2$ are simple polynomial maps, whereas $\varphi_4^{-1}$ and $\varphi_5^{-1}$ involve QR decompositions.\footnote{This is inevitable. The inverse of any isometry from $\St(k,n)/\GL(k)$ or $\GL(n)/\P(k,n)$ to $\Gr_\pi(k,n)$ must necessarily involve the QR-decomposition \cite{GS}.}  While polynomial-time algorithms for computing QR decomposition of a full-rank matrix abound (this is why $\varphi_4^{-1}$ and $\varphi_5^{-1}$ can be evaluated on any specific input in polynomial-time), they invariably involve extracting square roots and branching. There is no known rational map that can replace an algorithm for QR decomposition; should one exists, it is expected to be exceedingly complicated so that composing it with a polynomial $f$ would likely result in the defining parameters blowing up to exponentially high degrees.

\section{NP-hardness for Stiefel manifolds and orthogonal group}\label{sec:stie}

As we mentioned in Section~\ref{sec:intro}, unconstrained \emph{quadratic} programming on $\V(1,n) = \mathrm{S}^{n-1}$ is polynomial-time \cite[Section~4.3]{vava}. On the other hand, it is an immediate corollary of Proposition~\ref{prop:cli2} that for any fixed $k \le n$, unconstrained \emph{quartic} programming on $\V(k,n)$ is NP-hard. In fact we may also deduce the same for $\O(n)$, though not by simply setting $\O(n) = \V(n,n)$ as $k$ is fixed:
\begin{corollary}[Stiefel and orthogonal quartic programming is NP-hard]
Let $k \in \mathbb{N}$ be fixed.  Unless $\mathrm{P} = \mathrm{NP}$, there is no $\mathrm{FPTAS}$ that is polynomial in $n$ for maximizing quartic polynomials over $\V(k, n)$ and $\O(n)$.
\end{corollary}
\begin{proof}
Let $\Gr(k,n)$ be the projection model \eqref{eq:iden}. We define the two maps
\[\begin{tikzcd}
	{\O(n)} & {\V(k,n)} & {\Gr(k,n),}
	\arrow["{\pi_1}", from=1-1, to=1-2]
	\arrow["{\pi_2}", from=1-2, to=1-3]
\end{tikzcd}\quad
\pi_1(Q) \coloneqq Q \begin{bmatrix} I_k \\ 0 \end{bmatrix},\quad \pi_2(Y) \coloneqq Y Y^\tp,
\]
noting that right multiplication by $\begin{bsmallmatrix} I_k \\ 0 \end{bsmallmatrix} \in \mathbb{R}^{n \times k}$ gives a matrix in $\V(k,n)$ comprising the first $k$ columns of $Q$. Let $f : \Gr(k,n) \to \mathbb{R}$ be a quadratic polynomial as in \eqref{eq:quad}. 
Then the functions $f\circ \pi_2 : \V(k,n) \to \mathbb{R}$ and $f \circ \pi_2 \circ \pi _1 : \O(n) \to \mathbb{R}$ are given by
\[
f\circ \pi_2(Y) = f(YY^\tp), \quad f \circ \pi_2 \circ \pi _1 (Q) = f\biggl(Q\begin{bmatrix} I_k & 0 \\ 0 & 0 \end{bmatrix} Q^\tp\biggr),
\]
i.e., quartic polynomial functions in $Y$ and $Q$ respectively. In particular $f \circ \pi_2 (Y)$ is given by the same expression in \eqref{eq:quadY} and  $f \circ \pi_2 \circ \pi_1$ is given by a nearly identical expression (with $q_{ij}$ in place of $y_{ij}$). By Theorem~\ref{thm:NP2}, we deduce that there is no $\mathrm{FPTAS}$ for maximizing quartic polynomials over $\V(k, n)$ and $\O(n)$.
\end{proof}

It is not surprising that quartic programming is NP-hard, given that it is already NP-hard over $\mathbb{R}^n$ \cite{MK}, as well as in a variety of other scenarios \cite{JLZ}. So it remains to find out if unconstrained \emph{cubic} programming on $\V(k,n)$ is NP-hard. We will show that it is but this cannot be easily deduced from any of our results up to this point --- as we saw in the proof above, these give quartic programs on $\V(k,n)$.

We will instead use a celebrated result of Nesterov \cite{nesterov2003}, which shows that $\alpha(G)$, the \emph{stability number} of an $(n-1)$-vertex undirected graph $G$, may be expressed as the maximum of a homogenous cubic polynomial $f$ over the $(n-1)$-sphere $\mathrm{S}^{n-1}$.  We will not need to know the actual expression for $f$ nor the definition of stability number,  just that
\begin{equation}\label{eqn:Nesterov}
\max_{\lVert x \rVert = 1} f(x) = \sqrt{1-\frac{1}{\alpha(G)}}
\end{equation}
and that this implies there is no $\mathrm{FPTAS}$ for maximizing $f$ over $\mathrm{S}^{n-1}$.
\begin{theorem}[Stiefel and orthogonal cubic programming is NP-hard]\label{thm:VO}
Let $k \in \mathbb{N}$ be fixed. Unless $\mathrm{P} = \mathrm{NP}$, there is no $\mathrm{FPTAS}$ that is polynomial in $n$ for maximizing cubic polynomials over $\V(k, n)$ or $\O(n)$.
\end{theorem}
\begin{proof}
Let $e_1 \in \mathbb{R}^n$ be the first standard basis vector. Then $\pi_1: \V(k,n) \to \mathrm{S}^{n-1}$, $Y \mapsto Ye_1$, takes $Y$ to its first column vector. Let $f : \mathrm{S}^{n-1} \to \mathbb{R}$ be a polynomial function.  Then $f \circ \pi_1 : \V(k,n) \to \mathbb{R}$ is a polynomial of the same degree.  If $f$ is Nesterov's cubic polynomial in \eqref{eqn:Nesterov},  then $f \circ \pi_1$ is a cubic polynomial on $\V(k,n)$ with the desired property.  For $\O(n)$, note that same argument also applies to  $\pi_1: \O(n) \to \mathrm{S}^{n-1}$, $Q \mapsto Qe_1$.
\end{proof}

Nesterov's construction also yields an independent proof of the NP-hardness result for $\Gr(1, n)$. We will record it below given that this is possibly the most important Grassmannian: As an abstract manifold, it is the $(n-1)$-dimensional real projective space $\mathbb{RP}^{n-1}$; and as we saw in Section~\ref{sec:onlyn}, it may also be regarded as the set of density matrices of pure states.
\begin{corollary}[Grassmannian quadratic programming is NP-hard III]
Unless $\mathrm{P} = \mathrm{NP}$, there is no $\mathrm{FPTAS}$ that is polynomial in $n$ for maximizing quadratic polynomials over $\Gr(1,n)$.
\end{corollary}
\begin{proof}
Again we will use the projection model in \eqref{eq:iden} for $\Gr(1,n)$.  Recalling that we write vectors as column vectors, let
\[
\pi: \mathrm{S}^{n-1} \to \Gr(1,n),\quad \pi(x,y) = \begin{bmatrix}
x \\
y
\end{bmatrix} \begin{bmatrix}
x^\tp & y
\end{bmatrix} = \begin{bmatrix}
x x^\tp & yx  \\
y x^\tp  & y^2
\end{bmatrix},
\]
where $x \in \mathbb{R}^{n-1}$ satisfies
\[
\biggl\lVert \begin{bmatrix} x \\ y \end{bmatrix} \biggr\rVert^2 \coloneqq x_1^2 + \dots + x_{n-1}^2 + y^2 = 1. 
\]
Let $f$ be Nesterov's cubic polynomial in \eqref{eqn:Nesterov} and define the optimization problem over $\mathrm{S}^{n-1}$,
\[
\max_{\lVert x \rVert^2 + y^2 = 1} g(x,y) \coloneqq \max_{\lVert x \rVert^2 + y^2 = 1}  f(x)y.
\]
So $g$ is a homogeneous quartic polynomial in the vector variable $\begin{bsmallmatrix} x \\ y \end{bsmallmatrix}$. But observe that we may also write $g$ in the form of a homogeneous quadratic polynomial $h$ in the matrix variable $P =  \begin{bsmallmatrix}
x x^\tp & yx  \\
y x^\tp  & y^2
\end{bsmallmatrix}$ since $\pi$ is surjective. In other words, we have $h \circ \pi = g$ and therefore
\begin{align*}
\max_{P\in \Gr(1,n)} h(P) =  \max_{\|(x, y)\| = 1} f(x)y &= \max_{\lVert x \rVert\le 1} f(x)\sqrt{1-\lVert x \rVert^2} = \max_{\lVert x \rVert = 1,\; 0\le  t \le 1} f(tx)\sqrt{1-t^2} \\
&= \max_{\lVert x \rVert = 1,\; 0\le  t \le 1} f(x)t^3\sqrt{1-t^2} = c\max_{\lVert x \rVert = 1} f(x)
\end{align*}
for some constant $c > 0$. This shows that there is no $\mathrm{FPTAS}$ for maximizing quadratic polynomials over $\Gr (1,n)$.
\end{proof}

We establish an analogue of Proposition~\ref{prop:redu} in order to show that the NP-hardness results above apply to the quotient model of the Stiefel manifold and to the noncompact Stiefel manifold.
\begin{lemma}\label{lem:stiefel}
There exist polynomial-time computable diffeomorphisms $\psi_1:  \O(n)/\O(n-k) \to \V(k,n)$ and $\psi_2:  \GL(n)/\P_1(k,n) \to \St(k,n)$.
\end{lemma}
\begin{proof}
For the quotient manifold $\O(n)/\O(n-k)$, we regard $\O(n-k)$ as the subgroup
\[
\biggl\{ \begin{bmatrix}
I_k & 0 \\
0 & Q_2
\end{bmatrix} \in \O(n) : Q_2 \in \O(n-k) \biggr\}.
\]
With this convention, the map below is well-defined
\[
\psi_1:  \O(n)/\O(n-k) \to \V(k,n),\quad \psi_1(\lb Q \rb) = Q \begin{bmatrix} I_k \\ 0 \end{bmatrix}
\]
and invertible. The inverse map is
\[
\psi_1^{-1}:  \V(k,n) \to  \O(n)/\O(n-k),\quad \psi_1^{-1}(Y) = \lb Q_Y \rb,
\]
where $Q_Y$ is the eigenvector matrix in the eigenvalue decomposition $Y Y^\tp = Q_Y \begin{bsmallmatrix}
I_k & 0 \\
0 & 0 
\end{bsmallmatrix} Q_Y^\tp$, and although $Q_Y$ is not uniquely determined, its equivalence class $ \lb Q_Y \rb$ is. Hence  $\psi_1^{-1}$ is also well-defined.  It is routine to verify that $\psi_1^{-1}$ is indeed the inverse of $\psi_1$. The same argument in the proof of Proposition~\ref{prop:redu} shows that $\psi_1$ is a diffeomorphism.  Computation of $\psi_1$ is simply taking a submatrix while computation of $\psi^{-1}_1$ involves eigenvalue decomposition, but both are clearly computable in polynomial-time.

In the last paragraph, $\lb \, \cdot \, \rb=\lb \, \cdot \, \rb_{\O}$; in this paragraph $\lb \, \cdot \, \rb=\lb \, \cdot \, \rb_{\GL}$. We define
\begin{alignat*}{2}
\psi_2 &:  \GL(n)/\P_1(k,n) \to \St(k,n), & \psi_2(\lb X \rb) &= S \begin{bmatrix} I_k \\ 0 \end{bmatrix},  \\
\psi^{-1}_2 &: \St(k,n)   \to \GL(n)/\P_1(k,n),\quad &\psi^{-1}_2(S) &= \lb L_S \rb,  
\end{alignat*}
where $L_S \in \GL(n)$ is a unit lower-triangular matrix in an $LDL^\tp$ decomposition 
\[
S S^\tp = L_S D L_S^\tp
\]
with $D \in \mathbb{R}^{n\times n}$ a nonnegative diagonal matrix. Although $L_S$ is not unique, but different choices differ by an element in $\P_1(k,n)$, and thus $\psi_2^{-1}$ is well-defined. Again it is routine to verify that $\psi_2^{-1}$ is the inverse of $\psi_2$ and that $\psi_2$ is a diffeomorphism.  Computation of $\psi_2$ is simply taking a submatrix while computation of $\psi^{-1}_2$ involves $LDL^\tp$ decomposition, both computable in polynomial-time.
\end{proof}

We may now deduce that optimization over the quotient model of the compact Stiefel manifold is also NP-hard in the bit model.
\begin{corollary}[Compact Stiefel optimization is NP-hard]
Let $k \in \mathbb{N}$ be fixed. Unless $\mathrm{P} = \mathrm{NP}$, there is no $\mathrm{FPTAS}$ that is polynomial in $n$ for maximizing polynomial functions over $\O(n)/\O(n-k) $.
\end{corollary}
\begin{proof}
Let $f : \V(k,n) \to \mathbb{R}$ be a polynomial function.
The value of $f \circ \psi_1$ on $\lb Q \rb \in \O(n)/\O(k)$ is given by $f\bigl(Q\begin{bsmallmatrix} I_k \\ 0 \end{bsmallmatrix} \bigr)$ where $Q \in \O(n)$ is any representative of $\lb Q \rb$. In other words the parameters defining the map $f \circ \psi_1$ are the coefficients of $f$ repeated with some multiplicities, as we saw in Example~\ref{eg:quad}.
If there were an  $\mathrm{FPTAS}$ for maximizing polynomial functions on $\O(n)/\O(k)$, then we have an  $\mathrm{FPTAS}$ for maximizing $ f: \V(k,n) \to \mathbb{R}$, contradicting Theorem~\ref{thm:VO}. 
\end{proof}

Finally, we show that optimization over either the submanifold or quotient model of the non-compact Stiefel manifold, or the general linear group, is NP-hard in the oracle model.
\begin{corollary}[Noncompact Stiefel optimization is NP-hard]\label{cor:NPSt}
Let $k \in \mathbb{N}$ be fixed.  Unless $\mathrm{P} = \mathrm{NP}$,  there is no oracle-$\mathrm{FPTAS}$ that is polynomial in $n$ for 
maximizing polynomial functions over $\St(k,n)$, $\GL(n)/\P_1(k,n)$, or $\GL(n)$.
\end{corollary}
\begin{proof}
Let $\psi_2 : \GL(n)/\P_1(k,n) \to \St(k,n)$ be as defined in Lemma~\ref{lem:stiefel}. Let  $\theta_1 : \GL(n) \to \GL(n)/\P_1(k,n)$ be the quotient map taking a matrix to its coset in the quotient group. Let $\theta_2 :  \St(k,n) \to \V(k,n)$ be given by $\theta_2(S)  = Y_S $, where $Y_S \in \V(k,n)$ is the unique orthogonal factor in the condensed QR-decomposition $S = Y_S R_S$ in which $R_S \in \mathbb{R}^{k \times k}$ is upper triangular with positive diagonal. These three maps fit into a sequence:
\[
\GL(n) \xrightarrow{\theta_1} \GL(n)/\P_1(k,n) \xrightarrow{\psi_2} \St(k,n) \xrightarrow{\theta_2} \V(k,n).
\] 
Now let $f \circ \pi_1 : \V(k,n) \to \mathbb{R}$ be the cubic polynomial constructed in the proof of Theorem~\ref{thm:VO}.  Then
\begin{equation}\label{eq:three}
\begin{aligned}
f \circ \pi_1  \circ \theta_2 &: \St(k,n) \to \mathbb{R},\\
f \circ \pi_1  \circ \theta_2 \circ \psi_2 &: \GL(n)/\P_1(k,n) \to \mathbb{R},\\
f \circ \pi_1  \circ \theta_2 \circ \psi_2 \circ \theta_1 &:  \GL(n) \to \mathbb{R}
\end{aligned}
\end{equation}
are all polynomial function that can be computed in the same time complexity as $f \circ \pi_1$, up to a polynomial time factor. The required conclusion than follows from the nonexistence of $\mathrm{FPTAS}$ for maximizing over $\V(k,n)$ in Theorem~\ref{thm:VO}.
\end{proof}
Note that while Theorem~\ref{thm:VO} is proved for the bit model, Corollary~\ref{cor:NPSt} is proved in the oracle model. The reason is as in the commentary after the proof of Corollary~\ref{cor:St} --- note that $\theta_2$ appears in every function in \eqref{eq:three}, and evaluating $\theta_2$ involves a QR decomposition.

\section{NP-hardness of QP on the Cartan manifold}\label{sec:BT}

The set of postive definite matrices $\mathbb{S}^n_\pp$ has two natural metrics --- the standard Euclidean metric $\langle X, Y \rangle = \tr(XY)$ and the Riemannian metric $\mathsf{g}_A (X,Y) = \tr(A^{-1} X A^{-1} Y)$, defined for tangent vectors $X, Y \in \mathbb{S}^n$ and $A \in \mathbb{S}^n_\pp$. The former is rightly called the positive definite \emph{cone} since a cone is, by definition, a subset of an Euclidean space and endowed with the standard Euclidean metric (which is in fact required to define a dual cone). When endowed with the latter metric,   $(\mathbb{S}^n_\pp, \mathsf{g})$ is no longer a cone, but a natural model for $\El(\mathbb{R}^n)$, the set of all centered ellipsoids in $\mathbb{R}^n$, and $\mathsf{g}$ is the unique Riemannian metric that captures the geometry of ellipsoids.

This Riemannian manifold was first discussed by \'Elie Cartan at length in \cite[pp.~364--372]{Cartan1} and more briefly in \cite{Cartan2}. There are many modern expositions, most notably Lang's \cite{Lang} and \cite[Chapter~XII]{LangBook} but also Eberlein's \cite{Eber1,Eber2,EberBook} and Mostow's \cite{Mos} and \cite[Section~3]{MosBook}. All three authors credited \cite{Cartan1, Cartan2}, directly or indirectly, for the most substantive features of this manifold, with Mostow explicitly attributing it to Cartan, and we follow suit in our article.

By way of the following lemma, we will deduce the section title from deciding copositivity of a matrix, one of the oldest NP-hard problems in optimization \cite{MK}.
\begin{lemma}\label{lem:BT}
Let $A \in \mathbb{R}^{n \times n}$ and set
\[
\alpha_{ijkl} \coloneqq 
\begin{cases}
a_{ik} & i =j \text{ and } k = l,\\
0 & \text{otherwise}.
\end{cases}
\]
for all $i,j,k,l \in \{1,\dots,n\}$. We define the quadratic form $f : \mathbb{S}^n \times \mathbb{S}^n \to \mathbb{R}$ by
\[
f(X) = \sum_{i,j,k,l=1}^n \alpha_{ijkl} x_{ij} x_{kl}.
\]
If $f(X) \ge 0$ for all $X \in \mathbb{S}^n_\pp$, then $A$ is copositive.
\end{lemma}
\begin{proof}
This follows from
\[
f(X) = \sum_{i,k=1}^n a_{ik} x_{ii} x_{kk} = x^\tp A x
\]
with $x \coloneqq \diag(X)$. If $f(X) \ge 0$ for all $X \in \mathbb{S}^n_\pp$, then $f(X) \ge 0$ for all $X \in \mathbb{S}^n_\p$. Note that as $X$ runs over $\mathbb{S}^n_\p$, its diagonal $x$ runs over $\mathbb{R}^n_\p$. 
\end{proof}
\begin{corollary}[Cartan quadratic programming is NP-hard]\label{cor:car}
Unconstrained quadratic optimization over $\mathbb{S}^n_\pp$, whether regarded as the positive definite cone or the Cartan manifold, is NP-hard.
\end{corollary}
\begin{proof}
If we can minimize $f$ in Lemma~\ref{lem:BT} over $\mathbb{S}^n_\pp$, then we can decide if $f(X) \ge 0$ for all $X \in \mathbb{S}^n_\pp$, and thereby whether $A$ is copositive.
\end{proof}

At this point it is worth highlighting the fact that the choice of metric does not play a role in NP-hardness results in optimization. The computational tractability or intractability of computing $\max_{x \in \mathcal{M}} f(x)$ depends only on $\mathcal{M}$ as a set. The choice of Riemannian metric does of course affect the convergence rate of the algorithm. For example, when minimizing a strictly convex $f : \mathbb{R}^n \to \mathbb{R}$, the usual gradient descent in the standard Euclidean metric $\langle v, w\rangle \coloneqq v^\tp w$ is linearly convergent, but gradient descent in the Riemannian metric $\langle v, w\rangle_x \coloneqq v^\tp \nabla^2 f(x)^{-1} w$ is equivalent to Newton's method, which is quadratically convergent. But no Riemannian metric we put on $ \mathbb{R}^n $ will ever yield a polynomial-time algorithm for an NP-hard optimization problem unless $\mathrm{P} = \mathrm{NP}$.

We now prove an analogue of Proposition~\ref{prop:redu} required to extend Corollary~\ref{cor:car} to $\GL(n)/\O(n)$, the quotient model of the Cartan manifold. 
\begin{lemma}\label{lem:psd}
There exists a polynomial-time computable diffeomorphism $\rho: \GL(n)/\O(n) \to \mathbb{S}^n_\pp$.
\end{lemma}
\begin{proof}
The required map and its inverse are given by
\begin{alignat*}{2}
\rho &: \GL(n)/\O(n) \to \mathbb{S}^n_\pp, & \rho(\lb X \rb) &= X^\tp X, \\
\rho^{-1} &: \mathbb{S}^n_\pp \to \GL(n)/\O(n),\quad &\rho^{-1}(S) &= \lb R_S \rb,
\end{alignat*}
where $R_S$ is the unique upper-triangular matrix in the Cholesky decomposition $S = R_S^\tp R_S$. It is routine to verify that $\rho^{-1}$ is indeed the inverse of $\rho$; and it is computable in polynomial time since Cholesky decomposition is. That $\rho$ is a diffeomorphism follows the same argument used in the proof of Proposition~\ref{prop:redu}.
\end{proof}

For the next result, we will again have to resort to oracle complexity, as the commentary after the proof of Corollary~\ref{cor:St} about QR decomposition also applies to Cholesky decomposition.
\begin{corollary}[Cartan optimization is NP-hard]\label{cor:car2}
Unless $\mathrm{P} = \mathrm{NP}$,  there is no oracle-$\mathrm{FPTAS}$ that is polynomial in $n$ for maximizing polynomial functions over $\GL(n)/\O(n)$.
\end{corollary}
\begin{proof}
By Lemma~\ref{lem:psd}, $\rho$ and $\rho^{-1}$ are both polynomial-time computable. We may maximize the function $f : \mathbb{S}^n_\pp \to \mathbb{R}$ in Lemma~\ref{lem:BT} by querying the function $f \circ \rho : \St(k, n)/\GL(k) \to \mathbb{R}$ as an oracle and using $\rho^{-1}$ to map a solution in $ \St(k, n)/\GL(k)$ back to $ \mathbb{S}^n_\pp$ in polynomial-time. By Corollary~\ref{cor:car}, it must be NP-hard to maximize over $\GL(n)/\O(n)$.
\end{proof}

\section{LP is polynomial-time}\label{sec:linear}

The hardness results involving quadratic programming are the best possible in the sense that unconstrained linear programs are polynomial-time solvable for the Stiefel, Grassmann, and Cartan manifolds. They have closed-form solutions computable in polynomial time, easy exercises along the lines of those in \cite{GVL}, but recorded below for completeness.
\begin{lemma}[Unconstrained LP over Stiefel, Grassmann, and Cartan manifolds]
\begin{enumerate}[\upshape (a)]
\item \label{lem:LP-item1} For any $A \in \mathbb{R}^{n\times k}$,
\[
\max_{X \in \V(k, n)} \tr(A^\tp X) = \sum_{i=1}^k  \sigma_i
\]
is attained at $X =U V^\tp$ where $A = U\Sigma V^\tp $ is a condensed singular value decomposition with $U, V \in \V(n, k)$, $\Sigma = \diag(\sigma_1, \dots, \sigma_k ) \in \mathbb{R}^{k\times k}$, $\sigma_1 \ge  \dots \ge  \sigma_k \ge  0$.  

\item \label{lem:LP-item2} For any $A \in \mathbb{R}^{n\times n}$,
\[
\max_{P \in \Gr(k, n)} \tr(A^\tp P) = \sum_{i=1}^k \lambda_i
\]
is attained at $P = Q_k Q_k^\tp$ where $(A+A^\tp)/2 = Q \Lambda Q^\tp$ is an eigenvalue decomposition with $Q\in \O(n)$, $\Lambda = \diag(\lambda_1, \dots, \lambda_n) \in \mathbb{R}^{n\times n}$, $\lambda_1 \ge  \dots \ge  \lambda_n$.  

\item \label{lem:LP-item3} For any $A \in \mathbb{R}^{n\times n}$,
\[
\sup_{X \in \mathbb{S}^n_\pp} \tr(A^\tp X) =
\begin{cases}
0 & \text{if } -(A+A^\tp) \in \mathbb{S}^n_\p, \\
+\infty & \text{otherwise}.
\end{cases}
\]
\end{enumerate}
\end{lemma}
\begin{proof}
The problem in \eqref{lem:LP-item1} transforms into
\[
\max_{X \in \V(k, n)} \tr(A^\tp X) = \max_{X \in \V(k, n)} \tr(\Sigma^\tp X) = \max_{X \in \V(k, n)} \sum_{i=1}^k \sigma_i v_{ii}.
\]
For any $X \in \V(k, n)$, the diagonal entries $x_{ii} \le  1$, so the maximum is bounded by $\sum_{i=1}^k \sigma_i$. On the other hand, $V = I_k$ achieves this bound.

The problem in \eqref{lem:LP-item2} transforms into
\[
\max_{P \in \Gr(k, n)} \tr(A^\tp P) = \max_{P \in \Gr(k, n)} \tr\biggl(\biggl[\frac{A+A^\tp}{2}\biggr]^\tp P \biggr) = \max_{P \in \Gr(k, n)} \tr(\Lambda^\tp P) = \max_{P \in \Gr(k, n)} \sum_{i=1}^k \lambda_i p_{ii}.
\]
For any $P \in \Gr(k, n)$, the diagonal entries satisfy $0 \le  p_{ii} \le  1$, $p_{11} + \dots + p_{nn}  = k$. Under these constraints, the function $\sum_{i=1}^k \lambda_i p_{ii}$ must be upper-bounded by $\sum_{i=1}^k \lambda_i$: If one of $p_{ii}$ is not $1$, we may always change it to $1$ and decrease the value of some $p_{ii}$, $i > k$, and increase the function value. After at most $k$ such changes, we arrive at the value $\sum_{i=1}^k \lambda_i$. On the other hand, $P = I_k$ achieves this bound. 

For the problem in \eqref{lem:LP-item3}, if there is an $X\in \mathbb{S}^n_\pp$ with $\tr(A^\tp X) > 0$, then $\tr(A^\tp (\mu X)) = \mu \tr(A^\tp X) \to +\infty$ as $\mu \to  +\infty$. Otherwise, $\tr(A^\tp X)\leq 0$ for all $X \in \mathbb{S}^n_\pp$ and $X\to 0$ gives the least upper bound.
\end{proof}

\section{Conclusion}\label{sec:con}

In case one thinks that we might get around these intractability results by imposing convexity on the objective function. We remind the reader that for any compact manifold without boundary, which includes $\Gr(k,n)$, $\V(k,n)$, and $\O(n)$, the only continuous geodesically convex function is the constant function. Given the recent interests in convex optimization over metric spaces \cite{LLN24}, it is worth noting that the previous statement holds true for any metric space as well. Indeed it holds true for many manifolds that are noncompact \cite{Yau}. For manifolds that do support nonconstant convex functions,  they are often nowhere dense within the space of smooth functions \cite{YK24}.

Nevertheless, one positive way to view the hardness results in this article is that they show that current pursuits in manifold optimization, which are exclusively focused on local optimization, are most likely the right way to go. Short of having $\mathrm{P} = \mathrm{NP}$, manifold optimization cannot become a topic in global optimization like convex optimization.

\subsection*{Future work} Our constructions show that unconstrained cubic and quartic programming over $\O(n)$ and $\V(k,n)$ are NP-hard for all fixed $k \in \mathbb{N}$. For $k = 1$, quadratic programming over $\V(1,n)$ is well-known to be polynomial-time solvable. It is also common knowledge that the NP-hard quadratic assignment problem may be written as a linearly constrained quadratic optimization problem over $\O(n)$ \cite{pardalos1994}. These observations leave open the complexity of unconstrained quadratic programming over $\O(n)$ and $\V(k,n)$ for $k \ge 2$.

\bibliographystyle{abbrv}

\end{document}